\documentclass[10pt]{article}
\usepackage{amsmath,amssymb,amsthm}
\usepackage{color}
\usepackage{booktabs}
\usepackage{graphicx}
\usepackage{epstopdf}
\usepackage{algorithm,algorithmic}
\usepackage{verbatim}
\usepackage{url}
\usepackage{float}
 \usepackage{array}
 \usepackage{multirow}
 \usepackage{algorithmic}
\usepackage{algorithm}
\usepackage{enumerate}
\newcolumntype{P}[1]{>{\raggedright\arraybackslash}p{#1}}

\theoremstyle{plain}
\newtheorem{theorem}{Theorem}[section]

\newtheorem{proposition}{Proposition}[section]
\newtheorem{corollary}{Corollary}[section]

\newtheorem{example}{Example}[section]
\numberwithin{equation}{section}

\newcommand{\R}{\mathbb{R}}
\newcommand{\N}{\mathbb{N}}

\newcommand{\E}{\mathbb{E}}

\newcommand{\kpr}{{k^\delta_\mathrm{pr}}}
\newcommand{\ekpr}{{\bar{k}^\delta_\mathrm{pr}}}
\newcommand{\kst}{{k^\delta_\mathrm{st}}}
\newcommand{\ekst}{{\bar{k}^\delta_\mathrm{st}}}
\newcommand{\kdp}{k^\delta_{\mathrm{dp}}}
\newcommand{\xdp}{x^\delta_{\kdp}}
\newcommand{\omd}{\chi_{\Omega_\kappa}}
\newcommand{\omk}{\chi_{\tilde{\Omega}_\ekst}}

\title{A Probabilistic Oracle Inequality and Quantification of Uncertainty of a modified Discrepancy Principle for Statistical Inverse Problems}
\author{ Tim Jahn\thanks{Institut f\"ur Numerische Simulation und Hausdorff Center for Mathematics, University of Bonn, Germany (\texttt{jahn@ins.uni-bonn.de})}}

\begin{document}

\maketitle

\begin{abstract}
In this note we consider spectral cut-off estimators to solve a statistical linear inverse problem under arbitrary white noise. The truncation level is determined with a recently introduced adaptive method based on the classical discrepancy principle. We provide probabilistic oracle inequalities together with quantification of uncertainty for general linear problems. Moreover, we compare the new method to existing ones, namely early stopping sequential discrepancy principle and the balancing principle, both theoretically and numerically.

\end{abstract}

\section{Introduction}

In this note we take a closer look at a recently introduced modified discrepancy principle for solving an inverse problem by means of spectral cut-off. The problem of interest reads

\begin{equation}\label{in:eq1}
 Kx=y,
\end{equation}

where $K:\mathcal{X}\to \mathcal{Y}$ is a compact injective operator with dense range between infinite dimensional Hilbert spaces. The problem \eqref{in:eq1} is known to be ill-posed in the sense that $K$ is not continuously invertible on the whole space. This causes problems, since the exact right hand side $y^\dagger\in\mathcal{R}(K)\subset \mathcal{Y}$ is unknown and we just have access to component measurements corrupted by noise. For the abstract corrupted data we write

\begin{equation}
y^\delta = y^\dagger + \delta Z,
\end{equation}
 where $\delta>0$ is the noise level and $Z$ is centered white noise with finite second moments, i.e., it holds that
 
 \begin{enumerate}[(i)]
 	\item $\E[(Z,y)]=0$,
 	\item $\E[(Z,y)(Z,y')] = (y,y')$,
 	\item $(Z,y) \stackrel{d}{=} \frac{\|y\|}{\|y'\|}(Z,y')$
 \end{enumerate}
 
  	for all $y,y'\in\mathcal{Y}$. We have to give an approximation to the true solution $x^\dagger$ based on component measurements $(y^\delta,y_1),(y^\delta,y_2),...$ with $y_1,y_2,...\in\mathcal{Y}$. Under the above assumptions the forward operator $K$ can be fully described by its singular value decomposition. There exist orthonormal bases $(v_j)_{j\in\N}\subset \mathcal{X}$ and $(u_j)_{j\in\N}\subset \mathcal{Y}$ as well as a sequence $\sigma_1\ge \sigma_2 \ge ... >0$ converging to zero such that $Kv_j=\sigma_j u_j$ and $K^*u_j=\sigma_jv_j$ for all $j\in\N$. We define the spectral cut-off estimator of $x^\dagger$ via
  	
  	\begin{equation}
  	x_k^\delta :=\sum_{j=1}^k \frac{(y^\delta,u_j)}{\sigma_j}v_j
  	\end{equation}

with the truncation level $k$  to be determined from the data $y^\delta$ and the noise level $\delta$. Note that here we assume the noise level $\delta$ to be known in advance, see \cite{harrach2020beyond, 10.1093/imanum/drab098, mathe2017complexity} on how to estimate it in a general setting by a generic method.  It is well-known that spectral cut-off estimators have excellent theoretical properties, but may be infeasible for general problems in very high dimension due to the fact that the singular value decomposition is unknown in practice and has to be calculated with enormous costs. However, recently computationally efficient methods like the randomised singular value decomposition \cite{ito2019regularized} have been analysed in the context of regularisation theory and show the potential to allow the use of spectral cut-off estimators in practically relevant settings.
 Choosing a suitable truncation level $k$ is one of the main issues in regularisation theory, and plenty of different techniques have been analysed in the past. For general a priori error bounds depending on unknown properties of the exact solution $\hat{x}$ we refer to \cite{bissantz2007convergence}. Of special interest are adaptive (see \cite{goldenshluger1999pointwise, tsybakov2000best}) a posteriori methods which are strategies to choose $k$ only dependent on the noisy measurement $y^\delta$ and the noise level $\delta$.  Many of the parameter choice rules in statistical inverse problems are  adapted from classical statistical methods used for direct regressions problems, i.e., problems where $K$ is the identity in \eqref{in:eq1}. We name here empirical risk minimisation \cite{li2020empirical}, the balancing principle \cite{mathe2006regularization} and generalised cross validation \cite{lukas1993asymptotic}, which are based on Stein's unbiased risk estimation \cite{stein1981estimation}, Lepski's method \cite{lepskii1991problem} and cross validation \cite{wahba1977practical} respectively. Others have its roots in the classical deterministic regularisation theory, e.g., the famous discrepancy principle \cite{morozov1968error} or heuristic methods like the quasi-optimality criterion \cite{tikhonov1965use, kindermann2018quasi} or the L-curve method \cite{hansen1992analysis}.
 In \cite{jahn2021optimal} we proposed a modification of the discrepancy principle for a data-driven choice of $k$. Originally, the discrepancy principle has its root in the classical deterministic theory where one assumes that one has an absolute upper bound of the norm of noise $\|y^\delta-y^\dagger\|$.  It then follows the paradigm that the data $y^\delta$ should only be approximated up to the amount of noise, i.e., $k$ should be determined such that $\|Kx_k^\delta - y^\delta\| \approx \|y^\delta-y^\dagger\|$. In the white noise setting however, it holds that $\E\|y^\delta-y^\dagger\|^2 = \sum_{j=1}^\infty\delta^2\E(Z,u_j)^2 = \infty$ and therefore the classical discrepancy principle is not applicable. Consequently the discrepancy principle has to be adapted adequately to the white noise case. This has been done in the past by either pre-smoothing the problem \eqref{in:eq1}, see \cite{blanchard2012discrepancy, lu2014discrepancy}, or by working directly on a finite dimensional discrete problem \cite{vogel2002computational, blanchard2018early}. In the recent work \cite{jahn2021optimal} we proposed a modified discrepancy principle using discretisation. However, different to the above mentioned works here the discretisation dimension is treated as an additional regularisation parameter.
  Precisely, we discretise \eqref{in:eq1} and consider only the first $m$ components $(y^\delta,u_1),...,(y^\delta,u_m)$. Then the data error fulfills $\E\left[\left\|\sum_{j=1}^m(y^\delta-y^\dagger,u_j) u_j\right\|^2\right] = m\delta^2$ and the classical implementation of the discrepancy principle  yields

\begin{equation}\label{kmax}
 \kdp(m):=\left\{k\ge 0 ~:~\sqrt{\sum_{j=k+1}^m(y^\delta,u_j)^2}\le \tau \sqrt{m}\delta \right\},
\end{equation}

where $\tau>1$ is a fudge parameter. It remains to determine the discretisation level $m$. In \cite{jahn2021optimal} it was shown that the ultimate choice

\begin{equation}\label{kmax1}
  \kdp = \max_{m\in\N}\kdp(m)
\end{equation}

yields a convergent regularisation method, as the noise level $\delta$ tends to zero. In this article we give a more complete picture and generalise the main result from \cite{jahn2021optimal} to arbitrary $K$ and prove an oracle type inequality \cite{donoho1994ideal} with a controlled probability, instead of asymptotic convergence rates under source conditions (\cite{lavrentiev1962nekotorykh, mathe2008general}).  Oracle inequalities became popular in statistics and guarantee that after fixing a family of estimators (in our case spectral cut-off estimators, where the estimators are indexed by the truncation level $k$) one obtains up to a constant the optimal error for given exact solution $\hat{x}$. They depend strongly on the fact that under white noise one has sharp estimates for the data propagation error; the variance. Minimax Convergence rates over source sets on the other hand are classic in the deterministic theory, where one only knows an upper bound of the norm of the error, without any structural information. They guarantee optimal convergence in a worst-case fashion where the true solution is an element of an unknown source set. See the survey article \cite{cavalier2011inverse} for more details. We state here a special case of the general main result presented in the following section \ref{sec2}.

\begin{corollary}\label{cor1}
	Assume that there exist $q,c_q,C_q>0$ such that $C_q j^{-q}\ge \sigma_j^2 \ge c_q j^{-q}$ for all $j\in\N$ and assume that the white noise $Z$ has finite fourth moment $\E(Z,y)^4\le \|y\|^4 \gamma_4<\infty$. Then for all $\kappa\ge 3$ there holds
	
	$$\sup_{\substack{x^\dagger\in\mathcal{X}\\ \ekpr(x^\dagger)\ge k_\delta}}\mathbb{P}\left(\|\xdp - x^\dagger\| \le C_{\tau}\min_{k\in\N}\|x_k^\delta-x^\dagger\|\right)\ge 1 - \max\left(\frac{12}{\tau^2+2\tau-3},9\right)\sqrt{\frac{2(1+\gamma_4)}{\kappa}} - \frac{2(\gamma_4+1)\bar{C}_q}{\kappa},$$
	
	with $C_{\tau}:=\max\left(\sqrt{2}\left(\frac{\tau+1}{\tau-1} +1\right), \sqrt{2}+(2\tau+1)\sqrt{\frac{9^{1+q}(1+q)C_q}{c_q}4}\right)$ and $\bar{C}_q$ given in Section \ref{sec:proofs}.
\end{corollary}

The quantity $\ekpr(x^\dagger)$ will be defined in the next section and ensures that $x^\dagger$ can be distinguished sufficiently from the zero signal, relative to the noise level $\delta$. In particular, it holds that $\ekpr(x^\dagger)\to \infty$ as $\delta\to 0$ for all non-degenerated $x^\dagger$; for which exist arbitrary large $j\in\N$ with $(x^\dagger,u_j)\neq 0$. To understand this condition better, consider the extreme case $x^\dagger=0$. Here the minimal error could only be obtained for $k_{dp}^\delta = 0$, which obviously cannot occur with  probability tending to $1$ (as e.g. $\delta\to0$), since a finite fixed number of (random) measurements cannot be controlled with probability tending to $1$. Thus the corollary states that for mildly ill-posed problems we obtain  up to a constant the best possible error for spectral cut-off regularisation, uniformly over all signals with sufficiently many relevant components. 
 
 In Section \ref{sec3} we comment on points which were left open in \cite{jahn2021optimal}, e.g. on how to perform the maximisation over the infinite set $\N$ practically. We end the introduction by stressing that the main difference of the proposed modified discrepancy principle to existing methods is the minimality of the assumptions. In Theorem \ref{th1} and \ref{th2} the white noise is only assumed to have a finite second moment and also the singular values of $K$ are arbitrary. Numerical and theoretical comparisons between different regularisation methods for statistical inverse problems can be found in e.g. \cite{bauer2011comparingparameter, werner2018adaptivity, lucka2018risk}.  In section \ref{sec3} and \ref{sec5} we compare our new method to the parameter choice strategies to which it is most closely related, namely the early stopping discrepancy principle (\cite{blanchard2018early, blanchard2018optimal}) and the balancing principle (\cite{mathe2003geometry,bauer2005lepskij,mathe2006lepskii}). In particular we show that similar to the balancing principle also the modified discrepancy principle coincides with Lepski's method in the direct case $K=Id$ and discuss how the different viewpoint yield an early-stopping type implementation.

\section{Main Results}\label{sec2}
For a truncation level $k$ we decompose the error in two parts, a data propagation error (variance) and an approximation error (bias)

\begin{align*}
\|x_k^\delta-x^\dagger\|^2 &= \sum_{j=1}^k\left(\frac{(y^\delta,u_j)}{\sigma_j} - (x^\dagger,v_j)\right)^2 + \sum_{j=k+1}^\infty(x^\dagger,v_j)^2 = \sum_{j=1}^k\frac{(y^\delta-y^\dagger,u_j)^2}{\sigma_j^2} + \sum_{j=k+1}^\infty(x^\dagger,v_j)^2\\
 &=\|x_k^\delta-\E[x_k^\delta]\|^2 + \|\E[x_k^\delta] - x^\dagger\|^2.
\end{align*}

For the analysis it will be convenient to not consider only the strong error $\|x_k^\delta-x^\dagger\|$, but also the predictive (or weak) error $\|K(x_k^\delta-x^\dagger)\|$, for which we obtain

\begin{equation*}
\|K(x_k^\delta-x^\dagger)\|^2 = \sum_{j=1}^k(y^\delta-y^\dagger,u_j)^2+\sum_{j=k+1}^\infty(y^\dagger,u_j)^2;
\end{equation*}

note that $(y^\dagger,u_j)=\sigma_j(x^\dagger,v_j)$.
Thus we can split the error in two parts, one part which increases monotonically with $k$ and another which decreases monotonically with $k$. Now note that minimising such a sum of a monotonically increasing and a monotonically decreasing term is in essence equivalent to balancing the two terms. We thus define the following two truncation levels 

\begin{align}
\kpr(x^\dagger):=\min\left\{k\in\N_0~:~ \sum_{j=1}^k(y^\delta-y^\dagger,u_j)^2 \ge \sum_{j=k+1}^\infty(y^\dagger,u_j)^2\right\},\\
\kst(x^\dagger):=\min\left\{k\in\N_0~:~ \sum_{j=1}^k\frac{(y^\delta-y^\dagger,u_j)^2}{\sigma_j^2} \ge \sum_{j=k+1}^\infty (x^\dagger,v_j)^2\right\}.
\end{align}

 They balance the competing error terms in strong and predictive norm respectively. For the sake of self-containedness we formulate the aforementioned well-known result (see e.g. Lemma 3.2. in \cite{bauer2008regularization}) which states that $\kpr$ or $\kpr-1$ and $\kst$ or $\kst-1$ (up to a constant factor) minimise the corresponding error norms.

\begin{proposition}\label{prop1}
	For all $x^\dagger\in\mathcal{X}$ with $\kpr(x^\dagger)\ge 1$ there holds
	\begin{align}
	    \min_{k\in\N_0}\|K(x_k^\delta-x^\dagger)\| &\ge \frac{1}{\sqrt{2}}\min\left(\|K(x_\kpr^\delta - x^\dagger)\|,\|K(x_{\kpr-1}^\delta-x^\dagger)\|\right)\\
	    \min_{k\in\N_0}\left[\|x^\delta_k-x^\dagger\|\right] &\ge \frac{1}{\sqrt{2}} \min\left(\|x_\kst^\delta-x^\dagger\|,\|x_{\kst-1}^\delta-x^\dagger\|\right).
	\end{align}
\end{proposition}
\begin{proof}[Proof of Proposition \ref{prop1}]
	Note that the minimisers of $\|K(x_k^\delta-x^\dagger)\|^2$ and $\|K(x_k^\delta-x^\dagger)\|$ obviously coincide. We use the convention $\sum_{j=1}^0 = 0$. By monotonicity and Definition of $\kpr$ there holds for $k\ge \kpr$
	
	\begin{align*}
	\|K(x_k^\delta-x^\dagger)\|^2 &= \sum_{j=1}^k(y^\delta-y^\dagger,u_j)^2 + \sum_{j=k+1}^\infty(y^\dagger,u_j)^2 \ge \sum_{j=1}^k(y^\delta-y^\dagger,u_j)^2\\ &\ge \frac{1}{2}\left(\sum_{j=1}^\kpr(y^\delta-y^\dagger,u_j)^2 + \sum_{j=1}^\kpr(y^\delta-y^\dagger,u_j)^2\right)\ge \frac{1}{2}\left(\sum_{j=1}^\kpr(y^\delta-y^\dagger,u_j)^2 + \sum_{j=\kpr+1}^\infty(y^\dagger,u_j)^2\right)\\
	&=\frac{1}{2}\|K(x_\kpr^\delta-x^\dagger)\|^2\end{align*}

   and for $k\le \kpr-1$
   
   \begin{align*}
   \|K(x_k^\delta-x^\dagger)\|^2&=\sum_{j=1}^k(y^\delta-y^\dagger,u_j)^2 + \sum_{j=k+1}(y^\dagger,u_j)^2 \ge \sum_{j=k+1}^\infty(y^\dagger,u_j)^2 \ge\frac{1}{2}\left(\sum_{j=\kpr}^\infty(y^\dagger,u_j)^2 + \sum_{j=\kpr}^\infty(y^\dagger,u_j)^2\right)\\
    &\ge \frac{1}{2}\left(\sum_{j=1}^{\kpr-1}(y^\delta-y^\dagger,u_j)^2 + \sum_{j=\kpr}^\infty(y^\dagger,u_j)^2\right) = \frac{1}{2}\|K(x_{\kpr-1}^\delta-x^\dagger)\|^2,
   \end{align*}
	
	which proves the assertion for $\kpr$. The argumentation for the second assertion  is analogous.
	
\end{proof}

Note that $x_k^\delta$ and hence also $\kpr$ and $\kst$ are random quantities. In order to formulate our main results it is handy to have a deterministic quantity. We thus define

\begin{equation}\label{sec2:eq3}
\ekpr(x^\dagger):=\min\left\{k\in\N_0~:~\delta^2 k \ge \sum_{j=k+1}^\infty(y^\dagger,u_j)^2\right\}.
\end{equation}

Obviously $\ekpr$ is closely related to $\kpr$, since $k\delta^2=\E\left[ \sum_{j=1}^k(y^\delta-y^\dagger,u_j)^2\right]$ and thus $\ekpr$ balances the expected squared prediction error norm $\E\|K(x_k^\delta-x^\dagger)\|^2$.
We come to the first main result, a full oracle inequality for the prediction error. 

\begin{theorem}\label{th1}
	For all $\kappa\ge 3$ there holds
	
	$$\sup_{\substack{x^\dagger\in\mathcal{X}\\ \ekpr(x^\dagger)\ge \kappa}}\mathbb{P}\left(\|K(\xdp - x^\dagger)\| \le C_{\tau}\min_{k\in\N_0}\|K(x_k^\delta-x^\dagger)\|\right)\ge 1 - \max\left(\frac{12}{\tau^2+2\tau-3},9\right)\E\left[\left|\frac{1}{\kappa}\sum_{j=1}^{\lceil\frac{\kappa}{3}\rceil}\left((Z,u_j)^2-1\right)\right|\right]\to1,$$
	
	as $\kappa\to\infty$, with $C_{\tau}:=\sqrt{6}\sqrt{\frac{3}{2}\left(A_\tau+1\right)+B_\tau}$ and $A_\tau, B_\tau$ given below in Section \ref{sec:proofs}.
\end{theorem}

We want to point out the generality of the above result, which perfectly reflects the paradigm, that the discrepancy principle balances approximation and data propagation error in the image space in that it results, up to a constant, in the error of the weak balanced oracle. However, usually one is interested in the error in the strong norm. The result we obtain here will depend on the difference between the weak balanced oracle $\kpr$ and the strong balanced oracle $\kst$. We give an important relation between the two oracles in the next proposition.

\begin{proposition}\label{prop1a}
	It holds that $\kpr\le \kst$.
\end{proposition} 
\begin{proof}[Proof of Proposition \ref{prop1a}]
	We have
\begin{align*}
\sum_{j=1}^{\kpr-1}\frac{(y^\delta-y^\dagger,u_j)^2}{\sigma_j^2} &\le \frac{\sum_{j=1}^{\kpr-1}(y^\delta-y^\dagger,u_j)^2}{\sigma_{\kpr-1}^2} < \frac{1}{\sigma_{\kpr-1}^2}\sum_{j=\kpr}^\infty(y^\dagger,u_j)^2=\frac{1}{\sigma_{\kpr-1}^2}\sum_{j=\kpr}^\infty \sigma_j^2(x^\dagger,v_j)^2\\
&\le \sum_{j=\kpr}^\infty(x^\dagger,v_j)^2,
\end{align*}

where we used the definition of $\kpr$ in the second step. By definition of $\kst$ it follows that $\kst>\kpr-1$, hence $\kst\ge \kpr$.

\end{proof}

Proposition \ref{prop1a} already indicates that the strong error of the discrepancy principle might be suboptimal in the case that $\kpr$ is substantially smaller than $\kst$. This is indeed the case. The next theorem shows that the error in strong norm depends on the difference between $\kpr$ and $\kst$.

\begin{theorem}\label{th2}
	For $\kappa\ge 3$ there holds

	\begin{align}
	&\sup_{\substack{x^\dagger\in\mathcal{X}\\ \ekpr(x^\dagger)\ge \kappa}}\mathbb{P}\left( \|\xdp - x^\dagger\|  \le  C_\tau \left(\min_{k\in\N} \|x_k^\delta - x^\dagger\| + \sqrt{\sum_{j=\kpr}^\kst (x^\dagger,v_j)^2}\right)\right)\\
	 \ge &1- \max\left(\frac{12}{\tau^2+2\tau-3},9\right)\E\left[\left|\frac{1}{\kappa}\sum_{j=1}^{\lceil\frac{\kappa}{3}\rceil}\left((Z,u_j)^2-1\right)\right|\right]\to 1
	\end{align}
	
	as $\kappa\to\infty$, with $C_\tau:=\sqrt{2}\max\left(\frac{\tau+1}{\tau-1}+1,1+\sqrt{\frac{3}{8}}(3\tau+1)\right)$.
\end{theorem}

Under additional assumptions we can bound the probability in Theorem \ref{th1} and \ref{th2} explicitly.

\begin{proposition}\label{prop1b}
	Assume that for some $2<p\le 4$ the white noise has finite $p$-th moment, i.e. $\E|(Z,y)|^p=\gamma_p\|y\|<\infty$. Then
	
	$$\E\left[\left|\frac{1}{\kappa}\sum_{j=1}^\kappa((Z,u_j)^2 - 1)\right|\right] \le 2^{1-\frac{2}{p}} \left(\gamma_p + 1\right)^{\frac{2}{p}} \kappa^{\frac{2}{p}-1} .$$
	
\end{proposition}
\begin{proof}[Proof of Proposition \ref{prop1b}]
	By convexity and the Marcinkiewicz-Zygmund inequality (Corollary 8.2 in \cite{gut2013probability}) it holds that
	
	\begin{align*}
	\E\left[\left|\frac{1}{\kappa}\sum_{j=1}^\kappa((Z,u_j)^2 - 1)\right|\right] &\le \left(\E\left[\left|\frac{1}{\kappa}\sum_{j=1}^\kappa((Z,u_j)^2 - 1)\right|^\frac{p}{2}\right]\right)^\frac{2}{p}\le \left( \kappa^{1-\frac{p}{2}}\E\left|(Z,u_1)^2-1\right|^\frac{p}{2}\right)^\frac{2}{p}\\
	&\le \left(2^{\frac{p}{2}-1}(\gamma_p+1)1\right)^\frac{2}{p}\kappa^{\frac{2}{p}-1}\le 2^{1-\frac{2}{p}}(\gamma_p+1)^\frac{2}{p} \kappa^{\frac{2}{p}-1}.
	\end{align*}
	
\end{proof}

Unlike Theorem \ref{th1} the above Theorem \ref{th2} is not a proper oracle inequality, since the additional term is likely to dominate. In fact this perfectly illustrates in full generality one well-known drawback of the discrepancy principle, namely that it saturates. Note that this saturation is hidden when one considers convergence rates over source conditions for the discrepancy principle with spectral cut-off. For some operators we can show that the additional term will not dominate the total error, see e.g. Corollary \ref{cor1}, and we obtain a full oracle inequality for the strong error norm. The proofs of Theorem \ref{th1}, Theorem \ref{th2} and Corollary \ref{cor1} are deferred to Section \ref{sec:proofs}.

\section{Comparison to other parameter choice strategies}\label{sec3}

In this section we will compare our rule to two other parameter choice strategies, the balancing principle and the early-stopping discrepancy principle. We start with the balancing principle. This principle is based on ideas introduced by Lepski \cite{lepskii1991problem}. Lepski's method is originally formulated in a regression setting. A signal corrupted by white noise has to be recovered. This corresponds to our problem \eqref{in:eq1} with $K=Id$. We state the problem explicitly. Given noisy component measurements

\begin{equation}
 (y^\delta,u_j)=(y^\dagger,u_j) + \delta(Z,u_j),\quad j\in\N
\end{equation}

and the spectral cut-off estimators 

\begin{equation}
 y^\delta_k:=\sum_{j=1}^k(y^\delta,u_j) u_j
\end{equation}

find the truncation level $k\in\N$ such that the error $\|y^\delta_k-y^\dagger\|$ is optimal; in some sense. Lepski proposed the following rule 

\begin{equation}
 k_{Lep}^\delta:=\min\left\{k\in\N~:~\|y_m^\delta-y_k^\delta\|\le \kappa \sqrt{m}\delta,~\forall m>k\right\},
\end{equation}

with a fudge parameter $\kappa>1$. The idea was that $k$ should be large enough such that the differences would be dominated by the variance (note that $\E\|y_m^\delta-\E[y_m^\delta]\|^2 = \E\left[\sum_{j=1}^m(y^\delta-y^\dagger,u_j)^2\right] = m\delta^2$). In the above direct setting our approach \eqref{kmax1} and \eqref{kmax} would yield

\begin{equation}
k_{dp}^\delta:=\max_{m\in\N} k_{dp}^\delta(m) = \max_{m\in\N} \min\left\{0\le k \le m~:~{\sum_{j=k+1}^m(y^\delta,u_j)^2} \le \tau \sqrt{m}\delta\right\}.
\end{equation}

It is not hard to show that the two parameter choices actually coincide in this setting.

\begin{theorem}
	Let $\kappa=\tau$. Then $k_{Lep}^\delta=k_{dp}^\delta$ pointwise.
\end{theorem}
	\begin{proof}
		
		We observe that
		
		$$\sum_{j=k+1}^m(y^\delta,u_j)^2 = \|y^\delta_m-y^\delta_k\|^2$$
		
		and hence
		
		$$k_{dp}^\delta= \max_{m\in\N}\left\{0\le k\le m~:~\|y^\delta_m-y^\delta_k\| \le \tau \sqrt{m}\delta \right\}.$$
		
		Now let $k=k^\delta_{Lep}$. By definition, for all $m>k$ it holds that $\|y^\delta_m-y_k^\delta\|\le\tau\sqrt{m}\delta$. We deduce $k_{dp}^\delta(m)\le k$ for all $m\in\N$ and thus $k_{dp}^\delta\le k_{Lep}^\delta$. Now let $k=k_{Lep}^\delta-1$. By definition there exists $m>k$ with $\|y_m^\delta-y_k^\delta\|> \tau \sqrt{m}\delta$. Thus $k_{dp}^\delta(m)>k$ and therefore $k_{dp}^\delta\ge k+1=k^\delta_{Lep}$. The proof is finished.
	\end{proof}

So we see that both methods coincide for direct regression problems, i.e., when $K=Id$. When a forward operator $K$ is involved, the balancing principle for inverse problems deduced from Lepski's method however follows a different approach. There, the estimators are not compared in the weak norm, but in the strong norm. The variance of the spectral cut-off estimator in the strong norm is

$$\E\|x_m^\delta-\E[x_m^\delta]\|^2 = \sum_{j=1}^m\frac{\E(y^\delta-y^\dagger,u_j)^2}{\sigma_j^2} = \delta^2\sum_{j=1}^m\frac{1}{\sigma_j^2}.$$

 I.e., we have

\begin{equation}\label{bal}
k_{bal}^\delta:=\min\left\{k\ge 0 ~:~\|x_m^\delta-x_k^\delta\| \le \kappa \delta^2 \sum_{j=1}^m\frac{1}{\sigma_j^2},~\forall m>k\right\}.
\end{equation}
	
Here the choice of $\kappa$ is delicate. The reason for this is the following. In the direct setting we have 

$$\|y^\delta_m-\E[y^\delta_m]\|^2-\E\|y^\delta-\E[y^\delta_m]\|^2 = \sum_{j=1}^m\left((y^\delta-y^\dagger,u_j)^2-1\right) = \delta^2\sum_{j=1}^m\left((Z,u_j)^2-1\right),$$

which is an i.i.d. sum of unbiased random variables and hence a reverse martingale, see the proof of Proposition \ref{prop2} in Section \ref{sec:proofs} below. Therefore its whole trajectory, i.e., $(\|y^\delta_m-\E[y^\delta_m]\|^2-\E\|y^\delta-\E[y^\delta_m]\|^2)_{m\in\N}$ can be controlled. Unlike this, in the case of inverse regression we have

$$\|x^\delta_m-\E[x^\delta_m]\|^2 - \E\|x_m^\delta - \E[x_m^\delta]\|^2 = \delta^2\sum_{j=1}^m\frac{(Z,u_j)^2-1}{\sigma_j^2},$$

which is just a sum with independent summands of increasing variance. In particular, depending on the behaviour of the singular values $\sigma_j$ the single last summand may give significant contribution to the whole sum or even dominate the whole sum. Consequently, one considers only finitely many estimators in the \eqref{bal}, i.e. $m\le D$ where $D$ depends on the noise level $\delta$ and the spectrum of $K$ and $\kappa$ will also depend on $D$ and therefore on $\delta$. Typically one sets $\kappa \sim \log(\delta)^{-1}$ and then obtains optimal convergence in $L^2$ up to a logarithmic correction for polynomially and exponentially ill-posed problems under Gaussian noise. While the concrete dependence of $\kappa$ on $\delta$ is usually tailored to obtain optimal convergence in $L^2$ for Gaussian noise, it is not directly clear whether a constant $\kappa$ could yield convergence in probability, cf. Theorem \ref{th1} and \ref{th2}. The following counter example shows that at least for exponentially ill-posed problems $\kappa$ has to depend on $\delta$.

\begin{example}\label{example}
	Let $\sigma_j^2 = e^{-j}$ and $x^\dagger=0$ and let $Z$ be Gaussian white noise. Assume that some $\kappa>1$ is fixed and set $m_\delta=\lceil \log(\delta^{-2})\rceil$. Let $k^\delta_{bal}$ be determined by the balancing principle \eqref{bal}. Then there exists $p_\kappa>0$ such that 
	
	\begin{equation}\label{ex}
	 \mathbb{P}\left(\|x^\delta_{k^\delta_{bal}} - x^\dagger\| \ge 1\right) \ge p_\kappa
	\end{equation}
	
	for all $0<\delta\le e^{-1}$. 
	
	We show that \eqref{ex} is fulfilled with 
	
	$$p_\kappa:= \mathbb{P}\left( |X|> e\kappa\right)$$
	
	where $X$ is a standard Gaussian. Indeed, we set 
	
	$$\Omega_\kappa:=\left\{|(Z,u_{m_\delta})|> e \kappa \right\}.$$
	
	By definition of $p_\kappa$ there holds $\mathbb{P}\left(\omd\right)=p_\kappa$.	It holds that 
	
	\begin{align*}
	\|x_{m_\delta}^\delta-x_{m_\delta-1}^\delta\|\omd &= \delta |(Z,u_{m_\delta})| e^\frac{m_\delta}{2}\omd > e \kappa \delta e^\frac{m_\delta}{2} \ge \kappa \delta \sqrt{\frac{e^{m_\delta+1}-1}{e-1}}  > \kappa \delta \sqrt{\sum_{j=1}^{m_\delta}e^j} 
		\end{align*}
	
	since $\frac{e^{m+1}-1}{e-1}\le e e^m$. Consequently we have 
	
	$$k^\delta_{bal}\omd \ge m_\delta \omd.$$
	
	Ultimately,
	
	\begin{align*}
	\|x^\delta_{k^\delta_{bal}} - x^\dagger\|\omd &\ge \frac{|(y^\delta,u_{m_\delta})|}{\sigma_{m_\delta}}\omd = \delta e^\frac{m_\delta}{2}|(Z,u_{m_\delta})|\omd> e\kappa \omd\ge 1\omd
	\end{align*}
	
	by definition of $m_\delta$, which finishes the proof of the assertion \eqref{ex}.
	
\end{example}

To put it in a nutshell, we saw that both methods are closely related to Lepski's method. The balancing principle has the advantage that the estimators are compared in the same norm in which one wants to have convergence, i.e., the strong norm. In case of the modified discrepancy principle the estimators are compared in weak norm. However, we have much worse control of the variance in the strong norm than in the weak norm, which makes the choice of the fudge parameter more delicate for the balancing principle and slightly deteriorates the rate. Finally, at the end of this section we discuss that the different viewpoint on Lepski's method provided by the modified discrepancy principle allows for another potential benefit related to early stopping. 

Note that a drawback of both methods is its high computational costs. A whole series of estimators needs to be computed and then compared to each other before the final choice can be made. A computational very attractive and still convergent method is the early stopping or sequential discrepancy principle \cite{blanchard2018early} which makes use of the computational simplicity of the plain discrepancy principle \cite{mathe2006discretized}.   Here, a maximal dimension $D$ is chosen first, respectively given by the measurement process. Then the classical discrepancy principle is applied, but with the parameter $\tau$ set to one

\begin{equation}
 k^\delta_{es}:=\min\left\{ 0\le k \le D~:~\sqrt{\sum_{j=k+1}^D (y^\delta,u_j)^2}\le \sqrt{D}\delta \right\}.
\end{equation}

The name sequential or early stopping refers to the fact that it takes the estimator which was computed last. Note that usually the singular value decomposition is unknown and has to be approximated numerically. In particular, calculations of later singular vectors are more costy and less accurate than calculations of the first ones. Thus a big advantage is due to the fact that in order to compute $k^\delta_{es}$ one just needs the first $k^\delta_{es}$ singular vectors and values. This can be seen at follows. In an application we usually have a discretised $K\in\R^{D\times D'}$ with $D,D'$ large and right hand side $y_D^\delta\in\R^D$. Expressing $y_D^\delta$ in the singular basis of $K$ then gives $y_D^\delta=\sum_{j=1}^{D}(y^\delta,u_j)u_j$ and thus $\sum_{j=k+1}^D(y^\delta,u_j)^2 = \|y_D^\delta\|^2 - \sum_{j=1}^k(y^\delta,u_j)^2$. Note hereby that the total norm $\|y_D^\delta\|^2$ can be computed without knowledge of singular vectors.
Similar to $k_{dp}^\delta$ also $k_{es}^\delta$ mimics the predictive oracle $k^\delta_{pr}$. In \cite{blanchard2018early} oracle inequalities for $k_{es}^\delta$ are proven. Due to the large variance in the residuals they do not hold for very smooth solutions. Also higher orders of the error distribution are needed, see e.g., Theorem 4.2 in \cite{10.1093/imanum/drab098}. We will see in Section \ref{sec5} that $k_{es}^\delta$ tends to be numerically less stable than the other two methods. Finally, there is an obvious way to combine the early stopping discrepancy principle with the modified discrepancy principle in that one can use the early discrepancy principle to determine a maximum $m$ over which we maximise in \eqref{kmax}. I.e. one could consider

\begin{equation}
 k^\delta_{com}:=\max_{m\le k_{es}^\delta} k_{dp}^\delta(m)
\end{equation}

or, more generally and to avoid the instabilities of $k_{es}^\delta$ one could set for $\tau>\tau_{min}\ge 1$

\begin{equation}
 k^\delta_{com}:=\max_{m\le m_{max}}\min\left\{0\le k\le m~:~\sqrt{\sum_{j=k+1}^m(y^\delta,u_j)^2}\le \tau \sqrt{m}\delta \right\}
\end{equation}

with 

\begin{equation}
m_{max}=\min\left\{0\le k \le D~:~\sqrt{\sum_{j=k+1}^D (y^\delta,u_j)^2}\le \tau_{min} \sqrt{D}\delta\right\}.
\end{equation}

\section{Proofs}\label{sec:proofs}
A central tool for the proofs will be the following proposition already used in \cite{jahn2021optimal}, which allows to control the measurement error.

\begin{proposition}\label{prop2}
	For any $\epsilon>0$ and $\kappa\in\N$ there holds
	
	$$\mathbb{P}\left(\left|\sum_{j=1}^m(y^\delta-y^\dagger,u_j)^2-m\delta^2\right| \ge \varepsilon m\delta^2,\quad\forall m\ge \kappa\right) \le\frac{1}{\varepsilon} \E\left[\left|\frac{1}{\kappa}\sum_{j=1}^\kappa\left((Z,u_j)^2-1\right) \right|\right] \to 0$$
	
	as $\kappa\to\infty$.
\end{proposition}
\begin{proof}[Proof of Proposition \ref{prop2}]
	We give a short proof in order to keep the paper self-contained; see Proposition 3.1 in \cite{jahn2021optimal} for more details. It holds that
	
	$$\mathbb{P}\left(\left|\sum_{j=1}^m(y^\delta-y^\dagger,u_j)^2-m\delta^2\right| \le \varepsilon m\delta^2,\quad\forall m\ge \kappa\right)= \mathbb{P}\left(\sup_{m\ge \kappa} \left|\frac{1}{m}\sum_{j=1}^m X_j\right|\le \varepsilon\right),$$
	
	where $X_j:=\frac{(y^\delta-y^\dagger,u_j)^2}{\delta^2}-1 = (Z,u_j)^2-1$. Note that the right hand side of the above equation is independent of the noise level $\delta$. Since $(X_j)_{j\in\N}$ is i.i.d with $\E[X_1]=0$ and $\E|X_1|\le 2$ the sample mean $\left(\frac{1}{m}\sum_{j=1}^mX_j\right)_{m\in\N}$ is a reverse martingale. Thus the Kolmogorov-Doob-inequality (Theorem 16.2 in \cite{gut2013probability}) yields
	
	$$\mathbb{P}\left(\sup_{m\ge \kappa}\left|\frac{1}{m}\sum_{j=1}^m X_j\right|> \varepsilon\right)\le \frac{1}{\varepsilon} \E\left[\left|\frac{1}{\kappa}\sum_{j=1}^\kappa   X_j\right|\right]\to 0$$
	as $\kappa\to\infty$, by the law of large numbers and the fact that reverse martingales are uniformly integrable. 
\end{proof}

Now we can prove the main results. So let $x^\dagger$ be given with $\ekpr(x^\dagger)\ge \kappa$. We suppress the dependence on $x^\dagger$ of the balancing index $\ekpr(x^\dagger)$ in the following.	We carry out the analysis on the following sequence of events on which we have perfect control of the measurement error
	
	\begin{equation}\label{th1eq0}
	\Omega_\kappa:=\left\{\left|\sum_{j=1}^m(y^\delta-y^\dagger,u_j)^2-m\delta^2\right| \le \min\left(\frac{(\tau+1)^2}{4}-1,\frac{1}{3}\right) m \delta^2,\quad \forall m\ge \frac{\kappa}{3} \right\}.
	\end{equation}
	
	 By Proposition \ref{prop2} there holds 
	 
	 \begin{equation}\label{th1:eq0a}
	 \mathbb{P}\left(\Omega_\kappa\right) \ge 1-\max\left(\frac{12}{\tau^2+2\tau-3},9\right)\E\left[\left|\frac{1}{\kappa}\sum_{j=1}^{\left\lceil\frac{\kappa}{3}\right\rceil}\left((Z,u_j)^2-1\right)\right|\right].\end{equation}
	 
	 Obviously on $\Omega_\kappa$ we have
	 
	 \begin{align*}
	\frac{2}{3}m\delta^2\omd \le&\sum_{j=1}^m(y^\delta-y^\dagger,u_j)^2\omd \le \frac{4}{3}m\delta^2\omd,\\
	 	  &\sqrt{\sum_{j=1}^m(y^\delta-y^\dagger,u_j)^2}\omd \le\frac{\tau+1}{2}\sqrt{m}\delta\omd.
	 \end{align*}
	 
	 We first show that 
	 
	 \begin{equation}\label{th1:eq0}
	 (\kpr-1) \omd \ge \frac{\kappa}{3}\omd
	 \end{equation}
	  for all $\kappa\ge 3$. Indeed, since $\lceil\frac{\ekpr}{3}\rceil\ge \frac{\kappa}{3}$ by definition of $\Omega_\kappa$ we have that
	  
	  \begin{align*}
	  \sum_{j=1}^{\left\lceil\frac{\ekpr}{3}\right\rceil}(y^\delta-y^\dagger,u_j)^2 \omd &\le \frac{4}{3}\left\lceil \frac{\ekpr}{3}\right\rceil \delta^2\le (\ekpr-1) \delta^2 < \sum_{j=\ekpr}^\infty(y^\dagger,u_j)^2 < \sum_{j=\left\lceil \frac{\ekpr}{3}\right\rceil+1}^\infty(y^\dagger,u_j)^2
	  \end{align*}
	  
	  and thus $\kpr\omd>\left\lceil\frac{\ekpr}{3}\right\rceil\omd$ by definition, which proves the claim \eqref{th1:eq0}.
	 
	 \begin{proof}[Proof of Theorem \ref{th1}]
	We begin with an upper bound for $\kdp$. For $A_\tau = \left(\frac{\tau+1}{\tau-1}\right)^2$  there holds
	
	\begin{equation}\label{th1eq1}
	\kdp \omd\le A_\tau\kpr .
	\end{equation}
	To prove \eqref{th1eq1} we have to show that $\kdp(m)\le A_\tau \kpr$ for all $m\in\N$. This is clearly true for all $m\in\N$ with $\kdp(m)\le \kpr$. Now consider $m\in\N$ with $\kdp(m)>\kpr$.  By the defining relation of the discrepancy principle and the definition of $\kpr$ and $\omk$ there holds
	
	\begin{align*}
	\tau \sqrt{m} \delta \omd &\le  \sqrt{\sum_{j=\kdp(m)}^m(y^\delta,u_j)^2} \omd \le \sqrt{\sum_{j=\kdp(m)}^m(y^\delta-y^\dagger,u_j)^2}\omd + \sqrt{\sum_{j=\kdp(m)}^m(y^\dagger,u_j)^2}\omd\\
	&\le \frac{\tau+1}{2}\sqrt{m}\delta + \sqrt{\sum_{j=\kpr+1}^m(y^\dagger,u_j)^2}\le \frac{\tau+1}{2}\sqrt{m}\delta + \sqrt{\sum_{j=1}^{\kpr}(y^\delta-y^\dagger,u_j)^2}\omd\\
	&\le \frac{\tau+1}{2}\sqrt{m}\delta + \frac{\tau+1}{2}\sqrt{\kpr}\delta.
	\end{align*}
	
	We obtain 
	
	$$\sqrt{m}\delta \omd \le \frac{\tau+1}{\tau-1} \sqrt{\kpr}\delta = \sqrt{ A_\tau\kpr}\delta.$$
	
	This together with the obvious inequality $\kdp(m)\le m$ finishes the proof of the claim \eqref{th1eq1}.
	
	This shows stability of the method in the sense that we do not stop too late. To finish the proof we have to show that we do not stop too early. We claim that
	
	\begin{equation}\label{th1eq2}
	\sum_{j=\kdp+1}^{\kpr}(y^\dagger,u_j)^2 \omd\le B_\tau \kpr \delta^2
	\end{equation}
	
	with $B_\tau = \frac{(3\tau+1)^2}{4}$. To prove \eqref{th1eq2} let us consider $m=\kpr$. Then, by definition of the discrepancy principle,
	
	\begin{align*}
	\tau \sqrt{\kpr}\delta\omd &\ge \sqrt{\sum_{j=\kdp(\kpr)+1}^{\kpr} (y^\delta,u_j)^2}\omd \ge  \sqrt{\sum_{j=\kdp(\kpr)+1}^{\kpr}(y^\dagger,u_j)^2}\omd -\sqrt{\sum_{j=\kdp(\kpr)+1}^{\kpr} (y^\delta-y^\dagger,u_j)^2}\omd \\
	&\ge \sqrt{\sum_{j=\kdp(\kpr)+1}^{\kpr}(y^\dagger,u_j)^2}\omd - \frac{\tau+1}{2}\sqrt{\kpr}\delta.
	\end{align*}
	
	By monotonicity and with the convention $\sum_{j=m}^n=0$ for $m> n$ we obtain
	
	\begin{equation*}
	\sum_{j=\kdp+1}^{\kpr}(y^\dagger,u_j)^2 \omd\le \sum_{j=\kdp(\kpr)+1}^{\kpr}(y^\dagger,u_j) \omd\le \frac{(3\tau+1)^2}{4} \kpr \delta^2 = B_\tau \kpr \delta^2
	\end{equation*}
	
	and \eqref{th1eq2} is proven. We combine the two estimates \eqref{th1eq1} and \eqref{th1eq2}
	
	\begin{align}\notag
	\|K(\xdp-x^\dagger)\|^2\omd &= \sum_{j=1}^{\kdp}(y^\delta-y^\dagger,u_j)^2\omd + \sum_{j=\kdp+1}^\infty(y^\dagger,u_j)^2\omd\\\notag
	&= \sum_{j=1}^{\kdp}(y^\delta-y^\dagger,u_j)^2\omd  +\sum_{j=\kdp+1}^{\kpr}(y^\dagger,u_j)^2 \omd+ \sum_{j=\kpr+1}^\infty(y^\dagger,u_j)^2\omd\\\notag
	&\le \sum_{j=1}^{A_\tau \kpr}(y^\delta-y^\dagger,u_j)^2 \omd + B_\tau \kpr \delta^2\omd + \sum_{j=1}^\kpr(y^\delta-y^\dagger,u_j)^2\omd\\\label{th1:eq3}
	&\le \left(\frac{3}{2}\left(A_\tau+1\right)+B_\tau\right)\kpr \delta^2\omd.
	\end{align}
	
On the other hand we have

\begin{align*}
\|K(x_{\kpr-1}^\delta-x^\dagger)\|^2\omd &=\sum_{j=1}^{\kpr-1}(y^\delta-y^\dagger,u_j)^2\omd + \sum_{j=\kpr}^\infty(y^\dagger,u_j)^2\omd\\
                                         &\ge \frac{2}{3}(\kpr-1)\delta^2\omd\ge \frac{1}{3}\kpr \omd.
\end{align*}

and also

\begin{align*}
\|K(x^\delta_\kpr-x^\dagger)\|^2  &=\sum_{j=1}^{\kpr}(y^\delta-y^\dagger,u_j)^2\omd + \sum_{j=\kpr+1}^\infty(y^\dagger,u_j)^2\omd\\
                                   &\ge \frac{2}{3}\kpr \delta^2\omd.
\end{align*}

	Therefore Proposition \ref{prop1} gives
	
	$$\min_{k\in\N}\|K(x_k^\delta-x^\dagger)\|\omd\ge \frac{1}{\sqrt{2}}\min\left(\|K(x_\kpr^\delta - x^\dagger)\|,\|K(x_{\kpr-1}^\delta-x^\dagger)\|\right)\omd \ge \frac{1}{\sqrt{6}}\sqrt{\kpr}\delta\omd.$$

	Together with \eqref{th1:eq3} we obtain
	
	$$\|K(\xdp-x^\dagger)\|\omd \le \sqrt{\frac{3}{2}\left(A_\tau+1\right)+B_\tau}\sqrt{\kpr}\delta\omd \le C_\tau \min_{k\ge \N_0}\|K(x_k^\delta-x^\dagger)\|$$
	
	where
	
	\begin{equation}\label{th1eq3}
	C_\tau:=\sqrt{6}\sqrt{\frac{3}{2}\left(A_\tau+1\right)+B_\tau}.
	\end{equation}
	
	Theorem \ref{th1} then follows with \eqref{th1:eq0a}.
	
\end{proof}

\begin{proof}[Proof of Theorem \ref{th2}]
	We again resort on the event $\Omega_\kappa$ from \eqref{th1eq0}. Note that $\kst\omd\ge \kpr\omd\ge 3\kappa$ by Proposition \ref{prop1a}. We start by showing stability of the approach. We claim that

	\begin{equation}\label{th2:eq1}
	\sqrt{\sum_{j=1}^{\kdp}\frac{(y^\delta-y^\dagger,u_j)^2}{\sigma_j^2}}\chi_{\Omega_\kappa}\chi_{\left\{\kdp\ge \kst\right\}}\le \frac{\tau+1}{\tau-1}\sum_{j=\kdp}^\infty(x^\dagger,v_j)^2\omd \chi_{\left\{\kdp\ge \kst\right\}}.
	\end{equation}

	 So assume that $m\ge \kst$ with $k_{dp}^\delta(m)\ge \kst$. Then we have
	
	\begin{align*}
	\tau \sqrt{m}\delta \chi_{\Omega_\kappa} &< \sqrt{\sum_{j=\kdp(m)}^m(y^\delta,u_j)^2} \le \sqrt{\sum_{j=\kdp(m)}^m(y^\delta-y^\dagger,u_j)^2}\chi_{\Omega_\kappa} + \sqrt{\sum_{j=\kdp(m)}^m(y^\dagger,u_j)^2}\\
	&\le \frac{\tau+1}{2}\sqrt{m}\delta + \sigma_{\kdp(m)}\sqrt{\sum_{j=\kdp(m)}^\infty(x^\dagger,v_j)^2}\\
	\Longrightarrow \frac{\sqrt{m}}{\sigma_{\kdp(m)}}\omd &\le \frac{2}{\tau-1} \sqrt{\sum_{j=\kdp(m)}^\infty(x^\dagger,v_j)^2}.
	\end{align*}
	
	Since obviously $\kdp(m)\le m$ we conclude that
	
	\begin{equation}\label{th2:eq2}
	\frac{\sqrt{\kdp}}{\sigma_{\kdp}} \omd \le\frac{2}{\tau-1}\sqrt{\sum_{j=\kdp}^\infty(x^\dagger,v_j)^2},
	\end{equation}
	
	whenever $\kdp\ge \kst$. Then we have $\kdp\omd\ge\kappa\omd>\frac{\kappa}{3}\omd$ and by \eqref{th2:eq2} it follows that
	
	\begin{align*}
	\sqrt{\sum_{j=1}^{\kdp}\frac{(y^\delta-y^\dagger,u_j)^2}{\sigma_j^2}} \omd &\le \frac{1}{\sigma_{\kdp}} \sqrt{\sum_{j=1}^{\kdp}(y^\delta-y^\dagger,u_j)^2}\omd \le \frac{\tau+1}{2} \frac{\delta\sqrt{\kdp}}{\sigma_{\kdp}}\omd
	\le \frac{\tau+1}{\tau-1}\sqrt{\sum_{j=\kdp}^\infty(x^\dagger,v_j)^2}.
	\end{align*}
	
  This proves the assertion \eqref{th2:eq1}. We come to the approximation error. It holds that
	
	\begin{align*}
	\sqrt{\sum_{j=\kdp+1}^{\kpr-1}(x^\dagger,v_j)^2} \omd &\le \frac{1}{\sigma_{\kpr-1}}\sqrt{\sum_{j=\kdp+1}^{\kpr-1}(y^\dagger,u_j)^2}\omd \le \frac{1}{\sigma_{\kpr-1}} \left(\sqrt{\sum_{j=\kdp+1}^{\kpr-1}(y^\delta,u_j)^2}+\sqrt{\sum_{j=\kdp+1}^{\kpr-1}(y^\delta-y^\dagger,u_j)^2}\right)\omd\\
	&\le \frac{3\tau+1}{2} \frac{\delta\sqrt{\kpr-1}}{\sigma_{\kpr-1}}\omd\le \sqrt{\frac{3}{8}}(3\tau+1) \frac{\sqrt{\sum_{j=1}^{\kpr-1}(y^\delta-y^\dagger,u_j)^2}}{\sigma_{\kpr-1}}\\
	&\le \sqrt{\frac{3}{8}}(3\tau+1) \frac{\sqrt{\sum_{j=\kpr}^\infty(y^\dagger,u_j)^2}}{\sigma_{\kpr-1}}\le \sqrt{\frac{3}{8}}(3\tau+1)\sqrt{\sum_{j=\kpr}^\infty(x^\dagger,v_j)^2}
	\end{align*}
	
	and consequently
	
	\begin{align}\notag
	\sqrt{\sum_{j=\kdp+1}^\infty(x^\dagger,v_j)^2}\omd &\le \sqrt{\sum_{j=\kdp+1}^{\kpr-1}(x^\dagger,v_j)^2}\omd+\sqrt{\sum_{j=\kpr}^\infty(x^\dagger,v_j)^2}\\\label{th2:eq3}
	&\le\sqrt{\frac{3}{8}}(3\tau+1)\sqrt{\sum_{j=\kpr}^\infty(x^\dagger,v_j)^2} \le \sqrt{\frac{3}{8}}(3\tau+1)\left(\sqrt{\sum_{j=\kpr}^{\kst}(x^\dagger,v_j)^2}+\sqrt{\sum_{j=\kst+1}^\infty(x^\dagger,v_j)^2}\right).
	\end{align}
	
	We put the both estimations together to finish the proof. Note that it is easy to see from Proposition \ref{prop1} that 
	
	$$\min_{k\in\N_0}\|x_k^\delta-x^\dagger\|^2 \ge \frac{1}{2}\max\left(\sum_{j=1}^{\kst-1}\frac{(y^\delta-y^\dagger,u_j)^2}{\sigma_j^2},\sum_{j=\kst+1}^\infty(x^\dagger,v_j)^2\right).$$
	
	We treat the cases $\kdp<\kst$ and $\kdp\ge \kst$ separately. First by \eqref{th2:eq1}
	
	\begin{align*}
	\|\xdp-x^\dagger\|\omd\chi_{\{\kdp\ge \kst\}} &\le \left(\sqrt{\sum_{j=1}^{\kdp}\frac{(y^\delta-y^\dagger,u_j)^2}{\sigma_j^2}} +\sqrt{\sum_{j=\kdp+1}^\infty(x^\dagger,v_j)^2}\right)\omd\chi_{\{\kdp\ge \kst\}}\\
	&\le\left(\frac{\tau+1}{\tau-1} \sqrt{\sum_{j=\kdp}^\infty(x^\dagger,v_j)^2}+\sqrt{\sum_{j=\kdp+1}^\infty(x^\dagger,v_j)^2}\right)\omd\chi_{\{\kdp\ge \kst\}}\\
	&\le \left(\frac{\tau+1}{\tau-1}+1\right)\sqrt{\sum_{j=\kst}^\infty(x^\dagger,v_j)^2}\le \left(\frac{\tau+1}{\tau-1}+1\right)\left(\sqrt{(x^\dagger,v_\kst)^2}+\sqrt{\sum_{j=\kst+1}^\infty(x^\dagger,v_j)^2}\right)\\
	&\le \sqrt{2}\left(\frac{\tau+1}{\tau-1}+1\right)\left(\min_{k\in\N_0}\|x_k^\delta-x^\dagger\|+\sqrt{\sum_{j=\kpr}^\kst(x^\dagger,v_j)^2}\right).
	\end{align*}
	
	Second, by \eqref{th2:eq3}
	
	\begin{align*}
		\|\xdp-x^\dagger\|\omd\chi_{\{\kdp< \kst\}} &\le \left(\sqrt{\sum_{j=1}^{\kdp}\frac{(y^\delta-y^\dagger,u_j)^2}{\sigma_j^2}} +\sqrt{\sum_{j=\kdp+1}^\infty(x^\dagger,v_j)^2}\right)\omd\chi_{\{\kdp< \kst\}}\\
		&\le\left( \sqrt{\sum_{j=1}^{\kst-1}\frac{(y^\delta-y^\dagger,u_j)^2}{\sigma_j^2}}+\sqrt{\sum_{j=\kdp+1}^\infty(x^\dagger,v_j)^2}\right)\omd\\
		&\le \sqrt{2}\min_{k\in\N_0}\|x_k^\delta-x^\dagger\| + \sqrt{\frac{3}{8}}(3\tau+1)\left(\sqrt{\sum_{j=\kpr}^\kst(x^\dagger,v_j)^2}+\sqrt{\sum_{j=\kst+1}^\infty(x^\dagger,v_j)^2}\right)\\
		&\le\sqrt{2}\left(1+\sqrt{\frac{3}{8}}(3\tau+1)\right)\left(\min_{k\in\N_0}\|x^\delta_k-x^\dagger\| + \sqrt{\sum_{j=\kpr}^\kst(x^\dagger,v_j)^2}\right).
	\end{align*}
	
	Thus the proof of Theorem \ref{th2} is finished with
	
	$$C_\tau:=\sqrt{2}\max\left(\frac{\tau+1}{\tau-1}+1,1+\sqrt{\frac{3}{8}}(3\tau+1)\right).$$
	
\end{proof}

\begin{proof}[Proof of Corollary \ref{cor1}]
	
	This time we treat the cases $\kdp\le \kst$ and $\kdp>\kst$ separately. In the latter case, we obtain with \eqref{th2:eq1} that
	
	\begin{align*}
	\|\xdp-x^\dagger\|\omd\chi_{\{\kdp>\kst\}} &\le \left(\frac{\tau+1}{\tau-1}\sqrt{\sum_{j=\kdp}^\infty(x^\dagger,v_j)^2} + \sqrt{\sum_{j=\kdp+1}^\infty(x^\dagger,v_j)^2}\right)\omd\chi_{\{\kdp>\kst\}}\\
	&\le \left(\frac{\tau+1}{\tau-1}+1\right)\sqrt{\sum_{j=\kst+1}^\infty(x^\dagger,v_j)^2}\le \sqrt{2}\left(\frac{\tau+1}{\tau-1}+1\right)\min_{k\in\N_0}\|x_k^\delta-x^\dagger\|.
	\end{align*}
	
	For the former case, note that
	
	\begin{align*}
	\sqrt{\sum_{j=\kdp+1}^\kst(x^\dagger,v_j)^2}\omd&\le \frac{1}{\sigma_\kst}\left(\sqrt{\sum_{j=\kdp(\kst)+1}^\kst(y^\delta,u_j)^2}+\sqrt{\sum_{j=\kdp+1}^\kst(y^\delta-y^\dagger,u_j)^2}\right)\omd\\
	&\le \frac{1}{\sigma_\kst}\left(\tau\sqrt{\kst}\delta + \frac{\tau+1}{2}\sqrt{\kst}\delta\right)\le \frac{3\tau+1}{2}\frac{\sqrt{\kst}}{\sigma_\kst}\delta.
	\end{align*}
	
	Then
	
	\begin{align*}
	\|\xdp-x^\dagger\|\omd\chi_{\{\kdp\le \kst\}} &\le \sqrt{\sum_{j=1}^\kst\frac{(y^\delta-y^\dagger,u_j)^2}{\sigma_j^2}}\omd + \sqrt{\sum_{j=\kdp+1}^{\kst}(x^\dagger,v_j)^2}\omd + \sqrt{\sum_{j=\kst+1}^\infty(x^\dagger,v_j)^2}\\
	&\le \frac{\tau+1}{2}\frac{\sqrt{\kst}}{\sigma_\kst}\delta + \frac{3\tau+1}{2}\frac{\sqrt{\kst}}{\sigma_\kst}\delta + \sqrt{2}\min_{k\in\N_0}\|x_k^\delta-x^\dagger\|\\
	&=\left(2\tau+1\right)\frac{\sqrt{\kst}}{\sigma_\kst}\delta + \sqrt{2}\min_{k\in\N_0}\|x_k^\delta-x^\dagger\|.
	\end{align*}
	
	All we need to do to finish the proof is to control the strong oracle $\kst$. We will need the following two estimates for $k\in\N$ and $p>0$
	
	\begin{align*}
	\sum_{j=1}^k j^{p}\le  \int_1^{k+1}(x+1)^p{\rm d}x \le \frac{(k+2)^{1+p}}{1+p},\\
	 \sum_{j=1}^k  j^p \ge \int_1^{k+1}(x-1)^p{\rm d}x \ge \frac{k^{1+p}}{1+p}.
	\end{align*}
	
	 For any $k\ge1$  there holds

\begin{align}\notag
&\mathbb{P}\left(\left|\sum_{j=1}^k\frac{(y^\delta-y^\dagger,u_j)^2}{\sigma_j^2}- \delta^2\sum_{j=1}^k\frac{1}{\sigma_j^2}\right|\ge \frac{\delta^2}{2} \sum_{j=1}^k \frac{1}{\sigma_j^2}\right)\\\notag
\le &\frac{4\E\left[\left|\sum_{j=1}^k\frac{(y^\delta-y^\dagger,u_j)^2-\delta^2}{\sigma_j^2}\right|^2\right]}{ \delta^4 \left(\sum_{j=1}^k \frac{1}{\sigma_j^2}\right)^2} = \frac{4\delta^4\sum_{j=1}^k \frac{\E\left[\left|(Z,u_j)^2-1\right|^2\right]}{\sigma_j^4}}{\delta^4\left(\sum_{j=1}^k\frac{1}{\sigma_j^2}\right)^2}
\le \frac{8(\gamma_4+1)\sum_{j=1}^k\frac{1}{\sigma_j^4}}{\left(\sum_{j=1}^k\frac{1}{\sigma_j^2}\right)^2}\le \frac{\frac{8(\gamma_4+1)}{c_q^2}\sum_{j=1}^kj^{2q}}{\frac{1}{C_q^2}\left(\sum_{j=1}^k j^q\right)^2}\\\label{cor1:eq2}
\le & \frac{8(\gamma_4+1)C_q^2}{c_q^2} \frac{\frac{1}{2q+1} (k+2)^{2q+1}}{\left(\frac{1}{q+1} k^{q+1}\right)^2}\le \frac{8(\gamma_4+1)C_q^2(q+1)^23^{2q+1}}{c_q^2(2q+1)}\frac{1}{k}
=\frac{(\gamma_4+1)\tilde{C}_q}{k}
\end{align}

	with $\tilde{C}_q:=\frac{8C_q^2(q+1)^23^{2q+1}}{c_q^2(2q+1)}$. Similar to $\ekpr$ we define
	
	\begin{equation}
\ekst:=\min\left\{k\ge 0 ~:~ \delta^2\sum_{j=1}^k\frac{1}{\sigma_j^2} \ge \sum_{j=k+1}^\infty(x^\dagger,v_j)^2\right\}
	\end{equation}
	
	and set
	
	\begin{equation}
	 \tilde{\Omega}_{\ekst}:=\left\{\left|\sum_{j=1}^{\lceil\frac{\ekst}{3}\rceil}\frac{(y^\delta-y^\dagger,u_j)^2}{\sigma_j^2}- \delta^2\sum_{j=1}^{\lceil\frac{\ekst}{3}\rceil}\frac{1}{\sigma_j^2}\right|\ge \frac{\delta^2}{2} \sum_{j=1}^{\lceil\frac{\ekst}{3}\rceil} \frac{1}{\sigma_j^2},~\left|\sum_{j=1}^{3\ekst}\frac{(y^\delta-y^\dagger,u_j)^2}{\sigma_j^2}- \delta^2\sum_{j=1}^{3\ekst}\frac{1}{\sigma_j^2}\right|\ge \frac{\delta^2}{2} \sum_{j=1}^{3\ekst} \frac{1}{\sigma_j^2}\right\}.
	\end{equation}
	
	Note that $\ekst\ge \ekpr$, cf. Proposition \ref{prop1a}. On $\omk$ there holds that
	
	\begin{equation}\label{cor1:eq1}
	\frac{\ekst}{3}\omk\le \kst \omk -1 < \kst \omk \le 3\ekst \omk.
	\end{equation}
	
	We show the claim \eqref{cor1:eq1}. It holds that
	
	\begin{align*}
	\sum_{j=1}^{\lceil\frac{\ekst}{3}\rceil} \frac{(y^\delta-y^\dagger,u_j)^2}{\sigma_j^2}\omk &\le \frac{3\delta^2}{2}\sum_{j=1}^{\lceil\frac{\ekst}{3}\rceil}\frac{1}{\sigma_j^2}\le \frac{3\delta^2}{2} \left(\frac{\lceil\frac{\ekst}{3}\rceil}{\ekst-1} \sum_{j=1}^{\ekst-1} \frac{1}{\sigma_j^2}\right) \le \delta^2\sum_{j=1}^{\ekst-1}\frac{1}{\sigma_j^2}<\sum_{j=\ekst}^\infty(x^\dagger,v_j)^2\le \sum_{j=\lceil\frac{\ekst}{3}\rceil+1}^\infty(x^\dagger,v_j)^2
	\end{align*}
	
	where we used that $\sigma_j$ is monotonically decreasing in the second step and $\ekst\ge \kappa \ge 3$ in the third and in the last step. This shows that $\frac{\ekst}{3}\omk \le \kst-1$. For the second claim we can assume w.l.o.g. that $\sum_{j=3\ekst+1}^\infty(x^\dagger,v_j)^2>0$. Then we have
	
    \begin{align*}
    \sum_{j=1}^{3\ekst}\frac{(y^\delta-y^\dagger,u_j)^2}{\sigma_j^2}\omk \ge \frac{\delta^2}{2} \sum_{j=1}^{3\ekst}\frac{1}{\sigma_j^2}\omk \ge \frac{3\delta^2}{2}\sum_{j=1}^{\ekst}\frac{1}{\sigma_j^2}\ge \frac{3}{2}\sum_{j=\ekst+1}^\infty(x^\dagger,v_j)^2> \sum_{j=3\ekst+1}^\infty(x^\dagger,v_j)^2
    \end{align*}
	
	and we deduce that $\kst\omk\le 3\ekst$. Now we are ready to finish the proof of Corollary \ref{cor1}. It is
	
	\begin{align*}
	\delta^2\frac{\kst}{\sigma_\kst^2}\omk &\le \delta^2 \frac{3\ekst}{\sigma_{3\ekst}^2}\omk\le \frac{\delta^2}{c_q}(3\ekst)^{1+q}\omk\le \frac{9^{1+q}\delta^2}{c_q} \left\lceil\frac{\ekst}{3}\right\rceil^{1+q}\omk\le \frac{9^{1+q}\delta^2(1+q)C_q}{c_q} \sum_{j=1}^{\lceil \frac{\ekst}{3}\rceil}\frac{1}{\sigma_j^2}\omk\\
	&\le \frac{9^{1+q}(1+q)C_q}{c_q}2 \sum_{j=1}^{\lceil\frac{\ekst}{3}\rceil}\frac{(y^\delta-y^\dagger,u_j)^2}{\sigma_j^2}\omk\le \frac{9^{1+q}(1+q)C_q}{c_q}2\sum_{j=1}^{\kst-1}\frac{(y^\delta-y^\dagger,u_j)^2}{\sigma_j^2}\\
	&\le \frac{9^{1+q}(1+q)C_q}{c_q}4 \min_{k\in\N_0}\|x_k^\delta-x^\dagger\|^2.
	\end{align*}
	
	Putting all together we obtain
	
	\begin{align*}
	\|\xdp-x^\dagger\|\omk\omd &\le \max\left(\sqrt{2}\left(\frac{\tau+1}{\tau-1} +1\right), \sqrt{2}+(2\tau+1)\sqrt{\frac{9^{1+q}(1+q)C_q}{c_q}4}\right)\min_{k\in\N_0}\|x_k^\delta-x^\dagger\|
	\end{align*}
	
	and the proof follows with 
	
	$$\mathbb{P}\left(\omd\omk\right)\ge 1-\mathbb{P}\left(\omd\right)-\mathbb{P}\left(\omk\right) \ge 1 - \max\left(\frac{12}{\tau^2+2\tau-3},9\right)\frac{\sqrt{1+\gamma_4}}{\sqrt{\kappa}} - \frac{2(\gamma_4+1)\bar{C}_q}{\kappa},$$
	
	where we used \eqref{th1:eq0a} with Proposition \ref{prop1b} and \eqref{cor1:eq2}.
	
\end{proof}

\section{Numerical comparison}\label{sec5}

In this section we compare $k^\delta_{dp}$ with $k^\delta_{bal}$ and $k^\delta_{es}$ numerically. As examples we take four model problems from the open-source MATLAB tool-box Regutools \cite{hansen1994regularization}. Namely these are \texttt{phillips} and \texttt{deriv2} (mildly ill-posed) on the one hand and \texttt{gravity} and \texttt{heat} (severely ill-posed) on the other hand. They cover various settings and are discretisation of Volterra/Fredholm integral equations, which are solved by means of either quadrature rules or Galerkin methods. Note that the problems behind \texttt{deriv2} (the second derivative) and \texttt{heat} (the backwards heat equation) are also part of the recent numerical survey \cite{werner2018adaptivity}. The discretisation dimension for the examples is fixed at $D=5000$. As measurement noise we choose Gaussian white noise and the singular value decomposition is determined numerically with the function \texttt{csvd} from the toolbox. The simulations are run for noise levels $\delta=10^0,10^{-2},10^{-4},10^{-6}$.
The statistical quantities are computed by $1000$ independent Monte Carlo samples. Similar to the numerical survey \cite{werner2018adaptivity} we set $\tau=1.5$ in the definition of the discrepancy principle and $\kappa=4$ in the definition of the balancing principle, i.e,. for the balancing principle we choose a $\delta$ independent fudge parameter. Note that for $\kappa=4$ the probability $p_\kappa$ of a bad event from example \ref{example} is extremely small compared to only $100$ consecutive runs. For comparison we also calculated the (clearly inaccessible) optimal truncation level $k^\delta_{opt}:=\arg\min_{k\in\N}\|x^\delta_k-x^\dagger\|$. We express the sample mean of $e_{\cdot}:=\|x^\delta_{k^\delta_{\cdot}}- x^\dagger\|$  (where $\cdot={\rm dp}, {\rm bal}, {\rm es}, {\rm opt}$) together with the estimated standard deviation in tabular form in Tables 1-4. In Tables 5-8 we also present the statistics of the corresponding truncation levels $k^\delta_{\cdot}$ together with the optimal stopping time $k^\delta_{opt}=\arg\min \|x_k^\delta-x^\dagger\|$ and the weak and strong oracles $\kpr$ and $\kst$. It is clearly visible that $\kdp$ and $k^{\delta}_{\rm opt}$ are close to $\kpr$ and $\kst$ respectively, in accordance with Proposition \ref{prop1}. Interestingly, $k^\delta_{bal}$ tends to be close to $\kdp$ and hence $\kpr$ instead of $\kst$. This is probably due to the comparably large fudge parameter $\kappa=4$. Choosing a substantially smaller fudge parameter $\kappa$ however is less stable and we decided to stick to the choice from \cite{werner2018adaptivity}. We also see that the error for $k^\delta_{es}$ is quite large. From the Tables 5-8 it is not clear if this is due to very rare events where one substantially stops too late, or if there is a regular late stopping. We visualise the statistics of $\kpr$ as boxplots in Figure 1. While the median is relatively close to $\kpr$, in a substantial part of all runs it stops way too late. In one fourth of all runs the stopping happens later than the corresponding index of the upper border of the blue box. However, $k^\delta_{es}$ concentrates more and more as the noise level decreases. This result indicates that the early stopping discrepancy principle is better suited for comparably smaller noise levels.

\begin{table}[hbt!]
	\centering
	\caption{Sample mean and standard deviation of $e_{\cdot}$ for \texttt{phillips}.\label{tab:er-phillips}}
	\setlength{\tabcolsep}{4pt}
	\begin{tabular}{c|cccccccc|}
		\toprule
  $\delta$	& $e_{\rm dp}$ & $\mathrm{std}(e_{\rm dp})$ & $e_{\rm bal}$ & $\mathrm{std}(e_{\rm bal})$& $e_{\rm es}$& $\mathrm{std}(e_{\rm es})$ &$e_{\rm opt}$ &$\mathrm{std}(e_{\rm opt})$\\
 e0    & 1.0e0 & (1e-1)      & 1.1e0 &(2e-1) & 8.9e4 &(4e5) &6.3e-1 &(2e-1)\\
 e-2  & 7.6e-2 &(1e-3)   & 7.6e-2& (1e-3) & 5.7e2 &(2e3)&6.7e-2 &(1e-2)\\
 e-4     & 1.3e-2 &(3e-4)     & 1.3e-2& (3e-4) & 5.2e0 &(1e1)& 1.1e-2 &(2e-3)\\
 e-6 & 2.6e-3 &(3e-4) & 2.9e-3 &(1e-4) & 1.1e-1 &(3e-1)& 1.7e-3 &(2e-3)\\
\bottomrule
	\end{tabular}
\end{table}

\begin{table}[hbt!]
	\centering
	\caption{Sample mean and standard deviation of $e_{\cdot}$ for \texttt{deriv2}.\label{tab:er-deriv2}}
	\setlength{\tabcolsep}{4pt}
	\begin{tabular}{c|cccccccc|}
		\toprule
		$\delta$	& $e_{\rm dp}$ & $\mathrm{std}(e_{\rm dp})$ & $e_{\rm bal}$ & $\mathrm{std}(e_{\rm bal})$& $e_{\rm es}$& $\mathrm{std}(e_{\rm es})$ &$e_{\rm opt}$ &$\mathrm{std}(e_{\rm opt})$\\
		e0    & 8.3e0 &(3e1)      & 1.8e0 &(0) & 4.7e5 &(1e6) &1.7e0 &(2e-1)\\
		e-2  & 1.0e0 &(4e-1)   & 9.5e-1 &(8e-3) & 3.2e3 &(9e3)&8.5e-1 &(8e-2)\\
		e-4     & 4.7e-1 &(1e-2)     & 5.2e-1 &(3e-2) & 2.7e1 &(6.0e1)& 4.1e-1 &(2e-2)\\
		e-6 & 2.3e-1 &(4e-3) & 2.5e-1 &(4e-3) & 5.7e-1 &(1e0)& 2.0e-1 &(5e-3)\\
		\bottomrule
	\end{tabular}
\end{table}

\begin{table}[hbt!]
	\centering
	\caption{Sample mean and standard deviation of $e_{\cdot}$ for \texttt{gravity}.\label{tab:er-gravity}}
	\setlength{\tabcolsep}{4pt}
	\begin{tabular}{c|cccccccc|}
		\toprule
	$\delta$	& $e_{\rm dp}$ & $\mathrm{std}(e_{\rm dp})$ & $e_{\rm bal}$ & $\mathrm{std}(e_{\rm bal})$& $e_{\rm es}$& $\mathrm{std}(e_{\rm es})$ &$e_{\rm opt}$ &$\mathrm{std}(e_{\rm opt})$\\
	e0    & 4.6e0 &(6e-1)      & 5.0e0 &(8e-2) & 5.2e15 &(9e15) &3.2e0 &(6e-1)\\
	 e-2  & 8.9e-1 &(2e-1)   & 9.5e-1 &(2e-1) & 4.0e13 &(7e13)&6.5.8e-1 &(1e-1)\\
	 e-4     & 1.8e-1 &(2e-3)     & 1.8e-1 &(7e-3) & 5.2e11 &(9e11)& 1.2e-1 &(2e-2)\\
	 e-6 & 4.3e-2 &(1e-3) & 3.9e-2 &(6e-3) & 3.8e9 &(8e9)& 2.5e-2 &(4e-3)\\
	\bottomrule
\end{tabular}
\end{table}

\begin{table}[hbt!]
	\centering
	\caption{Sample mean and standard deviation of $e_{\cdot}$ for \texttt{heat}.\label{tab:er-heat}}
	\setlength{\tabcolsep}{4pt}
	\begin{tabular}{c|cccccccc|}
		\toprule
		$\delta$	& $e_{\rm dp}$ & $\mathrm{std}(e_{\rm dp})$ & $e_{\rm bal}$ & $\mathrm{std}(e_{\rm bal})$& $e_{\rm es}$& $\mathrm{std}(e_{\rm es})$ &$e_{\rm opt}$ &$\mathrm{std}(e_{\rm opt})$\\
		e0    & 1.7e1 &(6e0)      & 1.7e1 &(5e-1) & 5.6e7 &(4e8) &1.4e1 &(1e0)\\
		e-2  & 4.8e0 &(3e-1)   & 5.2e0 &(5e-2) & 1.1e5 &(5e5)&3.1e0 &(5e-1)\\
		e-4     & 6.4e-1 &(7e-2)     & 7.5e-1 &(6e-2) & 1.1e2 &(4e2)& 3.8e-1 &(3e-2)\\
		e-6 & 1.7e-1 &(3e-3) & 1.8e-1 &(5e-3) & 8.4e1 &(6e2)& 1.1e-1 &(8e-3)\\
		\bottomrule
	\end{tabular}
\end{table}

\begin{table}[hbt!]
	\centering
	\caption{Sample mean and standard deviation of $k_{\cdot}$ for \texttt{phillips}.\label{tab:k-phillips}}
	\setlength{\tabcolsep}{4pt}
\begin{tabular}{c|cccccccccccc|}
	\toprule
	$\delta$ & $k_{\rm dp}$ &$\mathrm{std}(k_{\rm dp})$& $k_{\rm bal}$ &$\mathrm{std}(k_{\rm bal})$& $k_{\rm es}$ &$\mathrm{std}(k_{\rm es})$& $k_{\rm opt}$ &$\mathrm{std}(k_{\rm opt})$&$k_{\rm pr}$ &$\mathrm{std}(k_{\rm pr})$ & $k_{\rm st}$ & $\mathrm{std}(k_{\rm st})$\\
	 e0    & 3.3 &(0.7)      & 2.9 &(0.5) & 47.5 &(63.7) &5 &(0.6)& 4.5 &(0.8) & 5.2 &(0.5)\\
	 e-2  & 7 &(0)   & 7 &(0) & 43.7 &(53.8)&8.1 &(1.2) & 7 &(0,1) & 9.4 & (0.7)\\
	 e-4     & 12 &(0)     & 12 &(0) & 48.5 &(51.2)& 15.9 &(2) & 12.1 &(0.6) & 17.4 &(1)\\
	 e-6 & 26.1 &(1) & 25.1 &(0.4) & 59.7 &(60.3)& 34.3 &(3.3) & 28 &(1) &35.3 & (1.2)\\
	\bottomrule
\end{tabular}
\end{table}

\begin{table}[hbt!]
	\centering
	\caption{Sample mean and standard deviation of $k_{\cdot}$ for \texttt{deriv2}.\label{tab:k-deriv2}}
	\setlength{\tabcolsep}{4pt}
	\begin{tabular}{c|cccccccccccc|}
		\toprule
		$\delta$ & $k_{\rm dp}$ &$\mathrm{std}(k_{\rm dp})$& $k_{\rm bal}$ &$\mathrm{std}(k_{\rm bal})$& $k_{\rm es}$ &$\mathrm{std}(k_{\rm es})$& $k_{\rm opt}$ &$\mathrm{std}(k_{\rm opt})$& $k_{\rm pr}$ &$\mathrm{std}(k_{\rm pr})$ & $k_{\rm st}$ & $\mathrm{std}(k_{\rm st})$\\
		e0    & 0.3 &(0.5)      & 0 &(0) & 43.5 &(72) &1.1 &(0.3)& 1 & (0) & 1.1 & (0.3)\\
		e-2  & 1.3 &(0.6)   & 1 &(0) & 39 &(62)&3.2 &(0.8) & 2 & (0.7) & 3.4 & (0.8)\\
		e-4     & 7 &(0.5)     & 46 &(54) & 48.6 &(59.5)& 12.4 &(1.7) & 8.7 & (0.8) & 14.6 & (1.5)\\
		e-6 & 32 &(1.1) & 27 &(1) & 53 &(38)& 53 &(4.6) & 39 & (1.5) &66.5 & (3.2) \\
		\bottomrule
	\end{tabular}
\end{table}

\begin{table}[hbt!]
	\centering
	\caption{Sample mean and standard deviation of $k_{\cdot}$ for \texttt{gravity}.\label{tab:k-gravity}}
	\setlength{\tabcolsep}{4pt}
	\begin{tabular}{c|cccccccccccc|}
		\toprule
	$\delta$ & $k_{\rm dp}$ &$\mathrm{std}(k_{\rm dp})$& $k_{\rm bal}$ &$\mathrm{std}(k_{\rm bal})$& $k_{\rm es}$ &$\mathrm{std}(k_{\rm es})$& $k_{\rm opt}$ &$\mathrm{std}(k_{\rm opt})$& $k_{\rm pr}$ &$\mathrm{std}(k_{\rm pr})$ & $k_{\rm st}$ & $\mathrm{std}(k_{\rm st})$\\
	 e0    & 4.3 &(0.5)      & 4 &(0.1) & 44.7 &(61) &5.9 &(0.8)& 4.2 &(0.4) & 5.7& (0.6)\\
	 e-2  & 8.7 &(0.5)   & 8.5 &(0.5) & 28.2 &(47.1)&10.5 &(0.8) & 8.2& (0.4) & 10.3 & (0.6)\\
	 e-4     & 13 &(0.1)     & 13 &(0.2) & 48.6 &(59.5)& 14.9 &(0.7) & 12.7 &(0.4) & 14.7  & (0.6)\\
	 e-6 & 17 &(0.1) & 17.4 &(0.5) & 43.7 &(49.4)& 19.3 &(0.7) & 17 & (0) & 19.2  & (0.6)\\
	\bottomrule
\end{tabular}
\end{table}

\begin{table}[hbt!]
	\centering
	\caption{Sample mean and standard deviation of $k_{\cdot}$ for \texttt{heat}.\label{tab:k-heat}}
	\setlength{\tabcolsep}{4pt}
	\begin{tabular}{c|cccccccccccc|}
		\toprule
		$\delta$ & $k_{\rm dp}$ &$\mathrm{std}(k_{\rm dp})$& $k_{\rm bal}$ &$\mathrm{std}(k_{\rm bal})$& $k_{\rm es}$ &$\mathrm{std}(k_{\rm es})$& $k_{\rm opt}$ &$\mathrm{std}(k_{\rm opt})$& $k_{\rm pr}$ &$\mathrm{std}(k_{\rm pr})$ & $k_{\rm st}$ & $\mathrm{std}(k_{\rm st})$\\
		e0    & 1.5 &(0.9)      & 0.1 &(0.3) & 53 &(66) &3.3 &(0.8)& 1.6 &(0.6) & 2.9 &(0.9)\\
		e-2  & 9.9 &(0.8)   & 9 &(0.1) & 41 &(53)&17 &(1) & 9.4 &(0.5) &13.8 &(0.8)\\
		e-4     & 26 &(0.6)     & 25 &(1.4) & 48 &(38)& 32 &(2) & 24.1 &(1.3) &29.4 &(1)\\
		e-6 & 51 &(0.4) & 50 &(0.5) & 72 &(42)& 67 &(2) & 49.4 &(0.9)& 62.4 & (1.5)\\
		\bottomrule
	\end{tabular}
\end{table}

\begin{figure}
	\label{box}
	\centering
	\setlength{\tabcolsep}{4pt}
	\begin{tabular}{c c}
		\includegraphics[width=.55\textwidth]{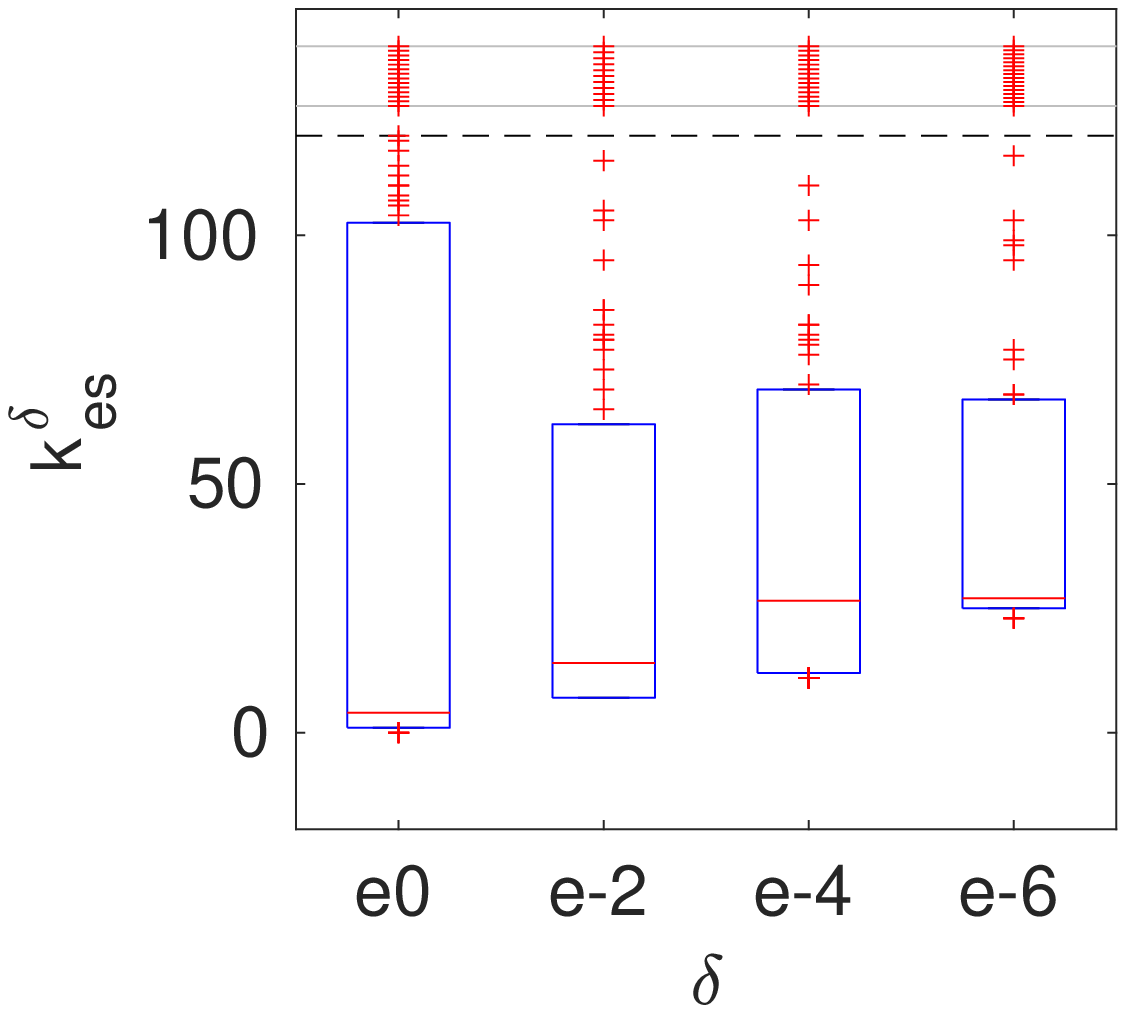} & \includegraphics[width=.55\textwidth]{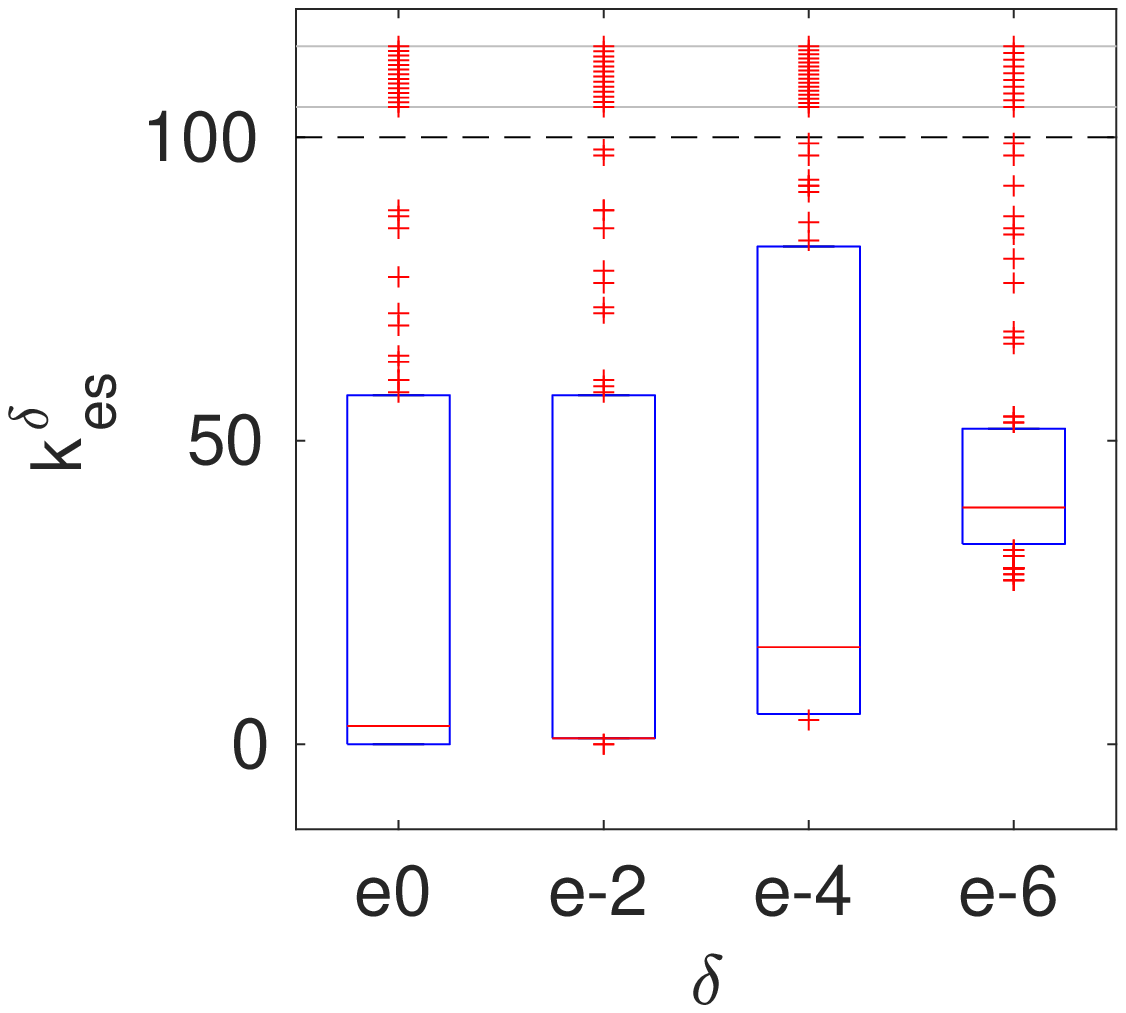}\\ 
		\includegraphics[width=.55\textwidth]{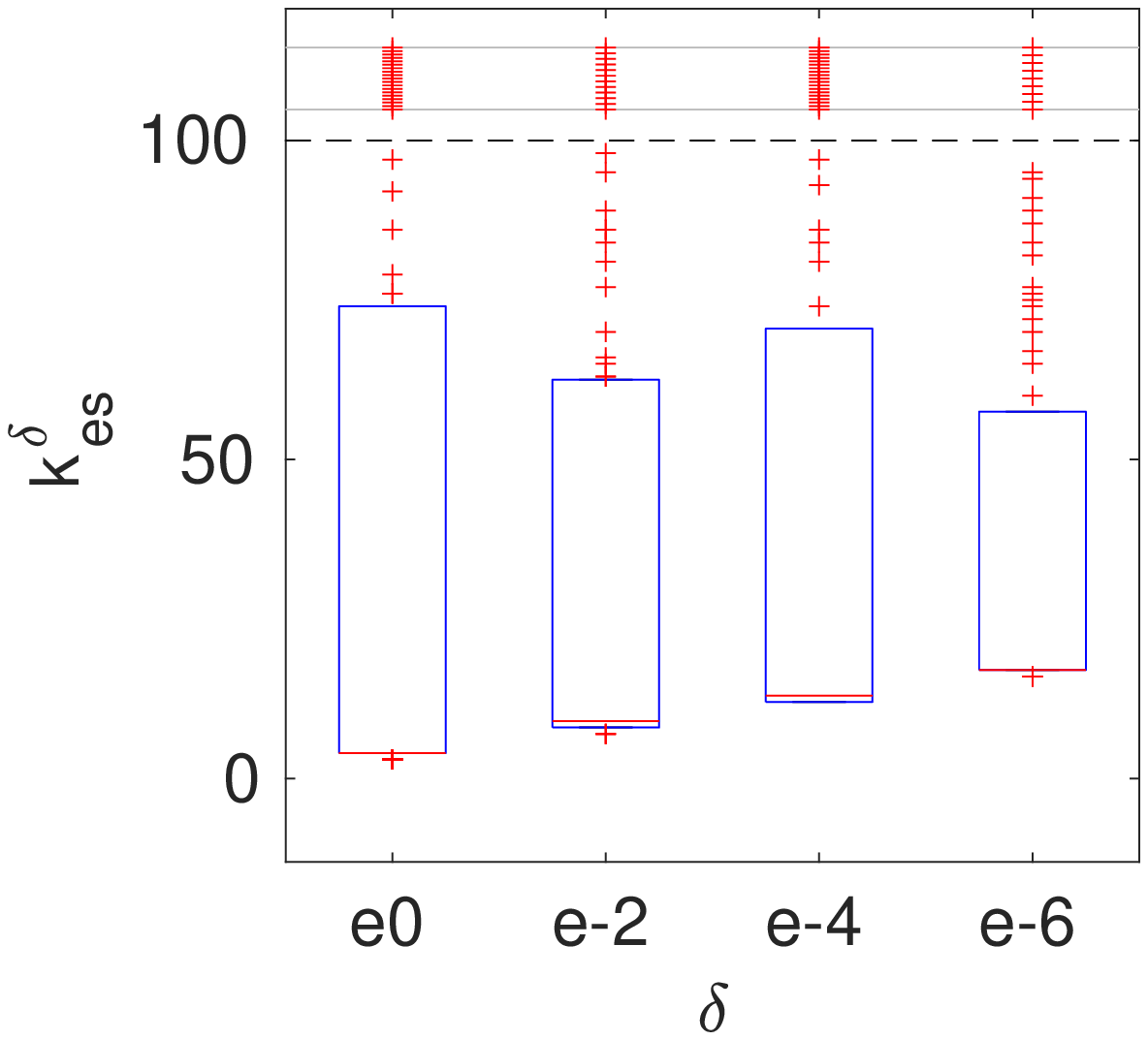} & \includegraphics[width=.55\textwidth]{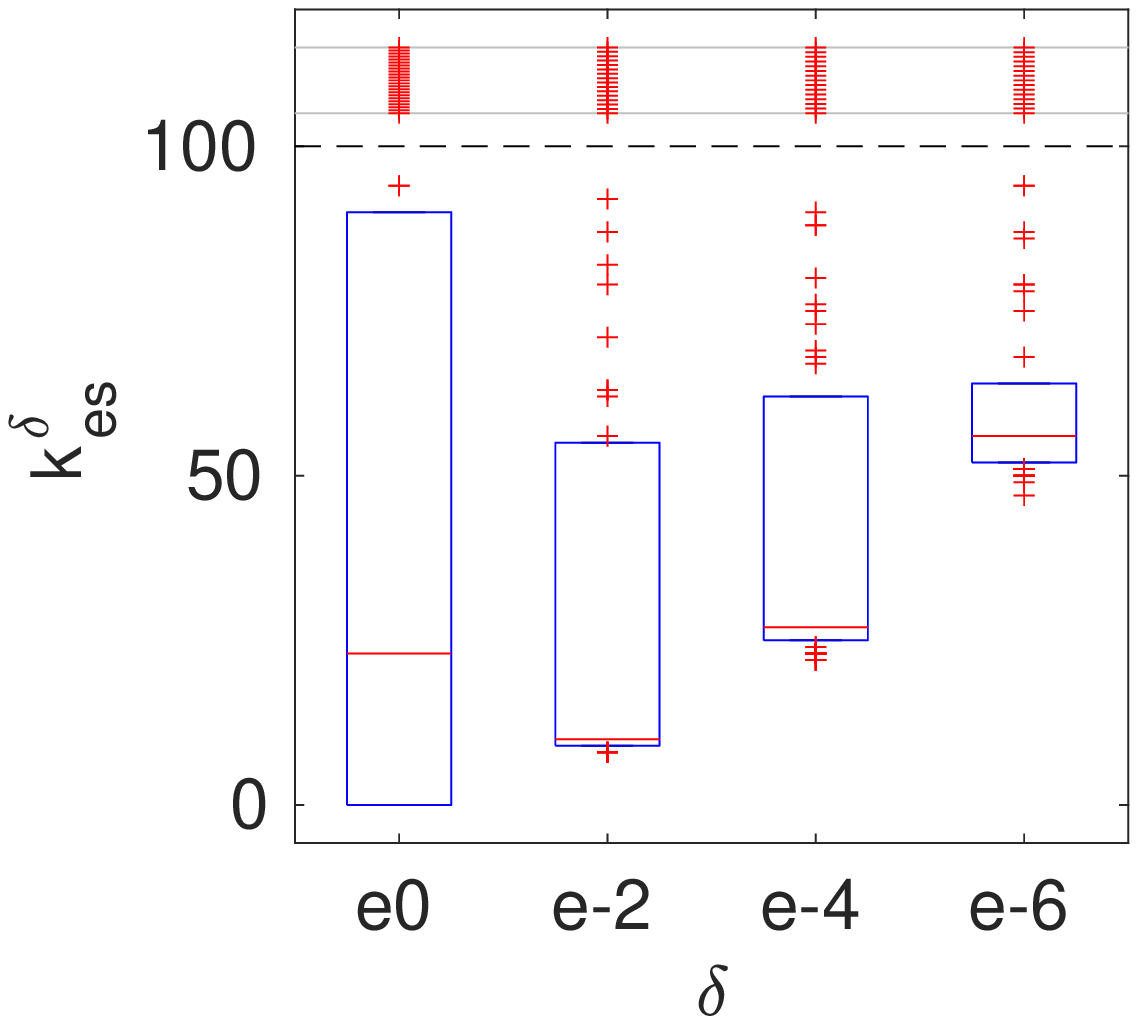}\\
	\end{tabular}
	\caption{Boxplots of the 100 realisations of $k^\delta_{es}$ for decreasing noise level $\delta$ and the different test problems (Upper left corner: \texttt{phillips}, upper right corner: \texttt{deriv2}, lower left corner: \texttt{gravity}, lower right corner: \texttt{heat}). On each blue box, the central red bar indicates the median, and the bottom and top edges of the box indicate the 25th and 75th percentiles. Red crosses depict outliers, which are data points falling outside the blue box. }
\end{figure}

\section{Acknowledgments}
Funded  by  the  Deutsche  Forschungsgemeinschaft (DFG, German Research Foundation) under Germany's Excellence Strategy  - GZ 2047/1, Projekt-ID 390685813. The author would like to thank Prof. Markus Rei\ss{} for pointing out possible relations of the modified discrepancy principle to Lepski's method.

\bibliographystyle{abbrv}
\bibliography{references}
\end{document}